\documentclass[reqno]{amsart}
\usepackage{amssymb,amsthm}
\usepackage[pdftex]{graphics}

\newcommand{\Z}{{\mathbb Z}}
\newcommand{\Q}{{\mathbb Q}}

\newcommand{\bpi}{\overline{\pi}}
\newcommand{\disc}{{\operatorname{disc}}}

\newtheorem{thm}{Theorem}

\newtheorem{lem}[thm]{Lemma}
\newtheorem{cor}[thm]{Corollary}

\title{An Application of the Dedekind-Hasse Criterion}
\author{F. Lemmermeyer}
\email{hb3@ix.urz.uni-heidelberg.de}
\address{M\"orikeweg 1, 73489 Jagstzell, Germany}

\begin{document}

\begin{abstract}
In this article we show how the Dedekind-Hasse criterion may be 
applied to prove a simple result about quadratic number fields 
that usually is derived as a consequence of the theory of ideals
and ideal classes.
\end{abstract}

\maketitle

\section*{Introduction}
Let $m$ be a squarefree integer, $K = \Q(\sqrt{m}\,)$ the quadratic
number field generated by the square root of $m$, $D_m$ its ring of
integers, and $\Delta = \disc K$ its discriminant. The following
result is called the 

\begin{thm}[Dedekind-Hasse Criterion]
The domain $D_m$ is a principal ideal domain if for all 
$\alpha, \beta \in D_m \setminus \{0\}$ with $\beta \nmid \alpha$ 
and $|N\alpha| \ge |N\beta|$ there exist  $\gamma, \delta \in D_m$
such that 
\begin{equation}\label{E1}
    0 < |N(\alpha \gamma - \beta \delta)| < |N\beta|. 
\end{equation}
\end{thm}

Actually, this is only a very special case of Dedekind's and Hasse's
result, who considered more generally arbitrary number fields and even 
general rings.

For squarefree integers $m$ as above we define the Gauss bound
$$ \mu_m = \begin{cases}
           \sqrt{-\Delta/3} & \text{ if } m < 0, \\
           \sqrt{ \Delta/5} & \text{ if } m > 0.
           \end{cases} $$  
In this note\footnote{This note was written around 1985.} 
we will show how to use the Dedekind-Hasse criterion for proving 
the following 

\begin{thm}\label{T1}
Assume that for all rational primes $p$ with $2 \le p \le \mu_m$ 
with $(\frac{\Delta}p) \ne -1$ there is an element $\pi \in D_m$ 
with $p = |N\pi|$, then $D_m$ is a PID.
\end{thm}

In the case $m < 0$, no prime $p < \mu_p$ can be a norm from $D_m$,
and we obtain the 

\begin{cor}
Assume that $m < 0$ and $(\frac{\Delta}p) = -1$ for all prime numbers
$p$ with $2 \le p \le \sqrt{-\Delta/3}$; then $D_m$ is a PID.
\end{cor}

In particular, $D_m$ is a PID for $-m = 1, 2, 3, 7, 11, 19, 43, 67, 163$.

It is also easy to see that Thm.~\ref{T1} holds whenever $D_m$ is a 
unique factorization domain; thus we find

\begin{cor}
A number ring $D_m$ is a UFD if and only if it is a PID.
\end{cor}

\section{Proof of the Main Theorem}

Since $N\beta \ne 0$, the condition (\ref{E1}) is equivalent to 
\begin{equation}\label{E2}
0 < |N(\xi \gamma - \delta)| < 1 \quad \text{ for all }
  \xi = \frac{\alpha}{\beta} = \frac{a+b\sqrt{m}}c \in K \setminus D_m;
\end{equation}
the exclusion of $\xi \in D_m$ comes from the condition that 
$\beta \nmid \alpha$. Write $\xi = \frac{a + b \sqrt{m}}c$ 
for integers $a, b, c$ with $c \ge 2$. Without loss of 
generality we may assume that $\gcd(a,b,c) = 1$.

\begin{lem}\label{L1}
It is sufficient to prove (\ref{E2}) for prime values of $c$.
\end{lem}

\begin{proof}
Assume that $c = c_1c_2$ is a factorization of $c$ with $c_1, c_2 \ge 2$.
Then at least one of $\frac{a+b\sqrt{m}}{c_1}$ or $\frac{a+b\sqrt{m}}{c_2}$
is not in $D_m$ unless $c_1 = c_2 = 2$, $m \equiv 1 \bmod 4$, and 
$a \equiv b \bmod 2$. We treat these cases separately.
\subsubsection*{$\frac{a+b\sqrt{m}}{c_1} \in K \setminus D_m$.} 
Assume that we can find $\gamma_1, \delta \in D_m$ with
      $$ 0 < \Big|N\Big(\frac{a+b\sqrt{m}}{c_1} 
                 \gamma_1 - \delta \Big) \Big| < 1. $$
      Setting $\gamma = c_2 \gamma_1$ we find
      $$ 0 < \Big|N\Big(\frac{a+b\sqrt{m}}{c} 
                 \gamma - \delta \Big) \Big| < 1 $$
       as desired.
\subsubsection*{$m \equiv 5 \bmod 8$, $c = 4$, $a \equiv b \equiv 1 \bmod 4$.}
Since $a^2 - mb^2 \equiv 4 \bmod 8$ there is an integer $\delta$ with
$a^2 - mb^2 = 8\delta + 4$. Set $\gamma = \frac{a-b\sqrt{m}}2$; then
$$ N(\xi \gamma - \delta) = N\Big( \frac12 \Big) = \frac14, $$
hence (\ref{E2}) is satisfied.
\subsubsection*{$m \equiv 1 \bmod 8$, $c = 4$, $a \equiv b \equiv 1 \bmod 4$.}
      Then $(\frac{\Delta}2) = +1$, hence there is nothing to prove
      in the case $\Delta < 0$ and $2 < \sqrt{-\Delta/3}$. In the 
      remaining cases there exists an element $\pi = \frac{x+y\sqrt{m}}2$
      with $|N(\pi)| = 2$:
      \begin{itemize}
      \item $m < 0$, $2 > \sqrt{-\Delta/3}$: then $m = -7$, and we can take
            $x = y = 1$.
      \item $m > 0$, $2 > \sqrt{\Delta/5}$: then $m = 17$, and we can take
            $x = 5$, $y = 1$.
      \item $m > 0$, $2 < \sqrt{\Delta/5}$: here the existence follows 
            from the assumption of the theorem.
      \end{itemize}
      Now set $\gamma = 1$ and
      $\delta = \frac12(\frac{a-x}2 + \frac{b-y}2 \sqrt{m}\,)$; then
      $ \xi \gamma - \delta = \frac{x+y\sqrt{m}}4 = \frac{\pi}2$,  hence
      $0 < |N(\xi  \gamma - \delta) = \big|N\big(\frac{\pi}2\big)\big| 
         = \frac12 < 1$ as desired.

This finishes the proof of Lemma \ref{L1}.
\end{proof}

Our next result is

\begin{lem}\label{L2}
It is sufficient to verify (\ref{E2}) for $c < \mu_m$.
\end{lem}

\begin{proof}
Since $\gcd(a,b,c) = 1$ there exist integers $d, e, f$ with 
$ad + be + cf = 1$. We distinguish two cases.
\subsection*{1. $m \equiv 2, 3 \bmod 4$.} By division with remainders in 
 the rational integers there exist integers $q, r$ with 
 $$ ae + mbd = cq+r, \quad \text{where} \quad
    \begin{cases}
            0 \le |r| \le \frac c2 & \text{ if } m < 0, \\
            \frac c2 \le |r| \le c & \text{ if } m > 0.
     \end{cases} $$
 Setting $\gamma = e + d \sqrt{m}$ and $\delta = q - f \sqrt{m}$ we find
 $$ \xi \gamma - \delta 
     = \frac{(ae + mbd - cq) + (ad + be + cf)\sqrt{m}}c 
     = \frac{r + \sqrt{m}}c, $$
 hence  $N(\xi \gamma - \delta) = \frac{r^2 - m}{c^2}$. 
\begin{itemize}
\item If $m < 0$ and $c > \sqrt{-\Delta/3} = \sqrt{-4m/3}$, then
      $$  0 <  N(\xi \gamma - \delta) = \frac{r^2 - m}{c^2}
            = \frac{r^2 + |m|}{c^2} < \frac{c^2 + 3c^2}{4c^2} = 1. $$
\item If $m > 0$ and $c > \sqrt{\Delta/5} = \sqrt{4m/5}$, then 
      $$ 0 <  N(\xi \gamma - \delta) = \frac{r^2 - m}{c^2} \quad
        \begin{cases} 
              \ge \frac{(c/2)^2 - m}{c^2} > -1, \\
               <  \frac{r^2}{c^2} \le 1. \end{cases} $$
\end{itemize}
This finishes the proof of Lemma \ref{L2} in the case $\Delta = 4m$.
\subsection*{2. $m \equiv 1 \bmod 4$.}
We claim that we can choose the integers $d, e, f$ with 
$ad + be + cf = 1$ in such a way that $d \equiv e \bmod 2$. In fact,
if $d \equiv e + 1 \bmod 2$ then either $c$ is odd or $c = 2$
(by Lemma \ref{L1}). If $c$ is odd we set $e' = e+c$ and $f' = f-b$;
then $ad + be' + cf' = 1$ and $d \equiv e' \bmod 2$. If $c = 2$
we must have $a \equiv b +1 \bmod 2$ (otherwise $\xi \in D_m$) and set
$d' = b+d$ and $e' = e-a$; then $ad' + be' + cf = 1$ and 
$d' \equiv e' \bmod 2$.

Now there are integers $q, r$ with $ ae + mbd = cq+r$, where we choose
$r$ in such a way that $q \equiv f \bmod 2$ and 
$$ \begin{cases}
            0 \le |r| \le c & \text{ if } m < 0, \\
            c \le |r| \le 2c & \text{ if } m > 0.
     \end{cases} $$
Setting $\gamma = \frac{e + d \sqrt{m}}2 \in D_m$ and
$\delta = \frac{q-f\sqrt{m}}2 \in D_m$ we verify (\ref{E2}) exactly
as in the case $m \equiv 2, 3 \bmod 4$.
\end{proof}

The final step in the proof of Thm. \ref{T1} is

\begin{lem}
It is sufficient to verify (\ref{E2}) in the case where
$c = p$ is prime with $a^2 - mb^2 \equiv 0 \bmod p$ and 
$(\frac{\Delta}p) \ne -1$.
\end{lem}

\begin{proof}
Assume that $c \nmid (a^2 - mb^2)$; then $a^2 - mb^2 = cq+r$ for integers
$q, r$ with $0 < |r| \le \frac c2$, and we set $\gamma = a - b \sqrt{m}$
and $\delta = q$. Then we obtain
$$ \xi \gamma - \delta = \frac{a^2 - mb^2}c -  = \frac rc \ne 0, $$
and the inequalities (\ref{E2}) are easily verified.

If $(\frac{\Delta}c) = -1$, the congruence $a^2 - mb^2 \equiv 0 \bmod c$
is not solvable for odd primes $c$. If $c = 2$, on the other hand,
then $(\frac{\Delta}c) = -1$ implies $\Delta = m \equiv 5 \bmod 8$,
and $a \equiv b \equiv 1 \bmod 2$ implies $\beta \mid \alpha$. 
\end{proof}

For the proof of Thm.~\ref{T1} it remains to take care of the
prime values $c = p < \mu_m$ with $p \mid \Delta$ or 
$(\frac{\Delta}p) = +1$. By assumption there is an element 
$\pi \in D_m$ with $|N\pi| = p$. For negative discriminants 
this is impossible, hence we only have to consider the case $m > 0$. 
We have to show that we can satisfy (\ref{E2}) for primes $c = p$ 
with $p \mid \Delta$.

\subsection*{1. The case $c = 2$.}
If $m \equiv 1 \bmod 4$, then $a^2 - mb^2 \equiv 0 \bmod 2$ implies
$a \equiv b \bmod 2$, hence $\beta \mid \alpha$ and $\xi \in D_m$.

If $m \equiv 2, 3 \bmod 4$, $a^2 - mb^2 \equiv 0 \bmod 2$ and 
$c < \sqrt{4m/5}$, then there is a $\pi = x + y \sqrt{m} \in D_m$
with $2 = |x^2 - my^2|$. We easily check that $a \equiv x$ and 
$b \equiv y \bmod 2$, and by setting $\gamma = 1$ and 
$\delta = \frac{a-x}2 + \frac{b-y}2 \sqrt{m} \in D_m$ we find
$$ |N(\xi \gamma - \delta)| = \Big| \frac{x^2 - my^2}4\Big| = \frac12 $$
as desired.

\subsection*{2. The case $c = p$ for odd primes $p$.}
Assume that $c = p$ is an odd prime, $m > 0$, $a^2 - mb^2 \equiv 0 \bmod p$
and $p < \sqrt{\Delta/5}$. By assumption there is a 
$\pi = \frac{x+y\sqrt{m}}2 \in D_m$ with $|N\pi| = p$. 
\subsubsection*{Case I. $p \nmid m$.} 
If we had $p \mid a$, then we also would have $p \mid b$ (and conversely), 
hence $\beta \mid \alpha$. Thus $p \nmid ab$. From $p = |\frac{x^2 - my^2}4|$
we deduce that $x^2 \equiv my^2 \bmod p$; since we also have 
$a^2 \equiv mb^2 \bmod p$ we must have 
$\frac{x}{2a} \equiv \pm \frac{y}{2b} \bmod p$. Replacing $y$ by $-y$
if necessary we may assume that, in this congruence, the plus sign holds;
letting $z$ denote an integer with $z \equiv \frac{x}{2a} \bmod p$ we find
$$ (a + b \sqrt{m}\,) z \equiv \frac{x+y\sqrt{m}}2 = \pi \bmod p. $$
Thus there is a $\delta \in D_m$ with $(a + b \sqrt{m}\,) z =  \pi + p \delta$.
We now set $\gamma = z$ and find
$$ \xi \gamma - \delta 
   = \frac{a+b\sqrt{m}}p z - \delta = \frac{x+y\sqrt{m}}{2p}, $$
which immediately shows that (\ref{E2}) is satisfied.

\subsubsection*{Case II. $p \mid m$.} 
Since $p \mid (a^2 - mb^2)$ we must have $p \mid a$. As before, $p \mid b$
would imply $\beta \mid \alpha$, hence $p \nmid b$. Since $m$ is squarefree,
we must have $p^2 \nmid (a^2 - mb^2)$, i.e., $\gcd(\frac{a^2-mb^2}p,p) = 1$.
Thus there exist integers $r, s$ with 
\begin{equation}\label{Est}
 \frac{a^2-mb^2}{p} \cdot r + ps = 1.
\end{equation}
Since $p \mid m$, the prime $p$ ramifies in $D_m$, hence $\pi \mid \sqrt{m}$
and therefore $\pi \mid a - b\sqrt{m}$ since $p \mid a$. Dividing (\ref{Est})
through by $\pi$ we find 
$$ \frac{a+b\sqrt{m}}p \cdot \frac{a-b\sqrt{m}}\pi \cdot r \pm \bpi s
    = \frac1{\pi}, $$
where $\bpi$ is the conjugate of $\pi$ and thus satisfies 
$\pi \bpi = N\pi = \pm p$. Setting $\gamma = \frac{a-b\sqrt{m}}{\pi} \cdot r$
and $\delta = \pm \bpi s$ we find that (\ref{E2}) is satisfied.

This finishes the proof of Theorem \ref{T1}.

\section{Applications}

Assume now that $D_m$ is a UFD. Then for all $\alpha, \beta \in D_m$
there exists a $\rho \in D_m$ with $(\rho) = (\alpha,\beta)$, and
$\lambda, \mu \in D_m$ with 
\begin{equation}\label{EBez}
   \alpha \lambda + \beta \mu = \rho. 
\end{equation}

In this section we will show that there is an algorithm for computing 
a Bezout representation (\ref{EBez}) using the Euclidean algorithm in $\Z$ 
and the prime elements $\pi$ in Theorem~\ref{T1} whose norms lie below
the Gauss bound.

In fact, given $\alpha$ and $\beta$ as above we can compute, as in the
proof of Theorem~\ref{T1}, elements $\gamma_0, \delta_0 \in D_m$ with
$\rho_1 = \alpha \gamma_0 - \beta \delta_0$ and $0 < |N\rho_1| < |N\beta|$.
If $\rho_1 \mid \beta$, then we also have $\rho_1 \mid \alpha$, and it
follows that $(\alpha,\beta) = (\rho_1)$, and that (\ref{EBez}) 
holds with $\lambda = \gamma_0$ and $\mu = \delta_0$.

If $\rho \nmid \beta$, then $|N\rho_1| < |N\beta|$ shows that we can
apply Thm.\ref{T1} to the pair $(\beta, \rho)$, and we can find
$\gamma_1, \delta_1 \in D_m$ with 
$$ \rho_2 = \beta \gamma_1 - \rho_1 \delta_1, \qquad 
    0 < |N\rho_2| < |N\rho_1|. $$
If $\rho_2 \mid \rho_1$, then $(\alpha,\beta) = (\rho_2)$,
and (\ref{EBez}) holds with $\lambda = -\gamma_0 \delta_1$ and 
$\mu = \gamma_1 + \delta_0 \delta_2$.

If $\rho_2 \nmid \rho_1$ we can apply Thm.~\ref{T1} again; since
the norm cannot decrease indefinitely, we eventually must find
that $\rho_n \mid \rho_{n-1}$. Then $(\alpha,\beta) = (\rho_n)$,
and by working backwards we find, in the usual way, the Bezout
elements $\lambda$ and $\mu$. 

\subsection*{Computing Prime Elements}

Assume that $p$ is a prime with $p > \mu_m$, and that we know an 
integer $x$ with $x^2 \equiv m \bmod p$. If $D_m$ is a UFD, then 
we can compute an element $\pi \in D_m$ with $|N\pi|  = p$ as
follows: set $\alpha = p$ and $\beta = x - \sqrt{m}$; then 
$(\pi) = (\alpha,\beta)$ for some $\pi \in D_m$ with norm $\pm p$.

\medskip\noindent{\bf Example.}
Let $m = 14$, $p = 137$, $x = 39$; then $\alpha = 137$ and 
$\beta = 39 - \sqrt{14}$. We find 
$\frac{137}{39 - \sqrt{14}} = \frac{39+\sqrt{14}}{11}$, hence
$a = 39$, $b = 1$, $c = 11$; we choose $d = 0$, $e = 12$, $f = -1$
and find $ad + be + cf = 1$. Moreover 
$ae = 468 = qc+r = 43 \cdot 11 - 5$, hence $q = 43$ and $r = -5$
(we choose $r$ in such a way that it minimzes $|r^2 - m|$), and
$\gamma_0 = 12$, $\delta_0 = 43 + \sqrt{14}$. Thus we find
\begin{align*}
  \rho_1 & = \alpha \gamma_0 - \beta \delta_0
           = 137 \cdot 12 - (39 - \sqrt{14}\,)(43 + \sqrt{14}\,) \\
         & = -19 + 4 \sqrt{14}. 
\end{align*}
Since $\frac{\beta}{\rho_1} = - 5 - \sqrt{14}$ we are already done:
$$ (137,39 - \sqrt{14}\,) = (-19 + 4 \sqrt{14}\,),$$
and in fact we have $N(-19 + 4 \sqrt{14}\,) = 137$.

\bigskip

\end{document}